\documentclass[12pt]{amsart}
\usepackage{graphicx}

\usepackage{amsmath}
\usepackage{amsfonts}
\usepackage{amssymb}

\usepackage{amsxtra}
\usepackage{amssymb,latexsym}
\usepackage[mathcal]{eucal}
\usepackage{amscd}

\usepackage{graphicx}
\usepackage{hyperref}
\usepackage{color}
\hypersetup{
    colorlinks,
    citecolor=blue,
    filecolor=blue,
    linkcolor=blue,
    urlcolor=blue
}

\newcommand{\red}[1]{{\textcolor{red}{#1}}}

\newtheorem{theorem}{Theorem}

\newtheorem{definition}{Definition}

\newcommand{\B}{\mathcal{B}}

\newcommand{\XBT}{(X, \mathcal{B}, \mu, T)}

\newcommand{\YCS}{(Y, \mathcal{C}, \nu, S)}

\def\qed{\relax\ifmmode\hskip2em \Box\else\unskip\nobreak\hfill$\Box$\fi}

\begin{document}

\title{The isomorphism problem for group actions}

\author{Matthew Foreman and Benjamin Weiss}

\address{
Mathematics Department\\
 University of California, Irvine\\
  Irvine, CA, }
 \email{mdforema@ad.uci.edu}

\address {Institute of Mathematics\\
 Hebrew University of Jerusalem\\
Jerusalem\\
 Israel}
\email{weiss@math.huji.ac.il}

\maketitle

\footnote{Foreman  partially supported by NSF grant DMS-2100367}

 \begin{abstract}
 We discuss the isomorphism problem for ergodic actions of locally compact groups.
In particular we show that the conjugacy relation is not Borel for ergodic measure preserving actions of indicable groups.
\end{abstract}
\bigskip

\section{Countable groups}

   The classic \textbf{isomorphism problem} in ergodic theory is the problem of classifying  measure
   preserving transformations $\XBT$ up to isomorphism. By an isomorphism of two such systems
   $\XBT$, $\YCS$ we mean a measurable invertible transformation $R:X \rightarrow Y$ such that
   $R \mu = \nu$ and $ SR = RT$.
 {Since all standard probability  spaces
   equivalent,  we can restrict attention to the group MPT of Lebesgue measure preserving transformation of the unit interval.} In this case
   $R$ is an element of MPT and the equation becomes $ S = RTR^{-1}$ which is the reason that this is also called the \textbf{conjugacy problem}.


   Over the years many striking positive results have been found. We will not discuss these any further here, a brief summary can be found in \cite {FRW} and \cite{GK}.
   From the perspective of descriptive set theory the isomorphism problem for MPT can be formulated in the following fashion. What is the nature of the set of pairs
   $(S, T) \in MPT \times MPT$ such that there exists some $R \in MPT$ satisfying $S = RTR^{-1}$?. If this set were to be  a Borel  subset that would mean that
   there is some fixed countable protocol, which when presented with a pair $(S,T)$ would determine whether or not $S$ is isomorphic to $T$. In \cite{H} G. Hjorth showed that
   the isomorphism relation in $MPT \times MPT$ is indeed not a Borel set but rather a \textbf{complete} analytic set. In order to explain what this is we recall the definition of a Borel reduction.

   \begin{definition} Let $X$ and $Y$  be Polish spaces. \\

    \textbf{(i)} A set $A \subset  X$  is \textbf{Borel reducible} to $B \subset Y$
if there is a Borel measurable function $f:X \rightarrow  Y$ such that for all $x \in X$ we have
$x \in  A$  if and only if $f(x) \in B$. \\

     \textbf{(ii)} An equivalence relation $E \subset X\times X$ is \textbf{Borel reducible} to an equivalence relation
     $F \subset Y\times Y$  if there is a Borel measurable function $f:X \rightarrow  Y$ such that
     $(u,v) \in E$ if and only if $(f(u) , f(v) \in F$.
     \end{definition}

    A subset $B \subset Y$ is \textbf{complete analytic} if for any Polish space $X$ and analytic
  subset $A \subset X$ there is a Borel reduction from $A$ to $B$. Since it is well-known that there are analytic sets that are not Borel a complete
  analytic set is not Borel.

   Hjorth's result used non-ergodic transformations in an essential way and this left open the basic isomorphism problem restricted to ergodic transformations. As
   is well known, the ergodic transformations form a dense $G_{\delta}$ in MPT and thus they too are a Polish space.
   This is the problem that we will be considering in the sequel. In \cite{FRW} we showed that the isomorphic relation for ergodic transformations forms a complete
   analytic set, in particular it is not Borel. This negative result explains perhaps why the isomorphism problem has been so difficult.

   In that paper we also used a simple Borel reduction to show that the isomorphism relation for ergodic actions of $\mathbb{R}$ is a complete analytic set.

   We turn next to countable groups in general. Just {as} MPT carries a natural Polish topology we can define a Polish topology on the space of all
   measure preserving actions of any countable group $G$. Such an action is simply a homomorphism from $G$ to MPT and a sequence of actions $S_g^{n}$  will converge to
   an action $S_g$ if and only if for each $g \in G$ we have that $S_g^{n}$ converges to $S_g$ in the topology of MPT. Once again we restrict to ergodic
   actions and ask for the nature of the isomorphism relation.

   A few years ago I. Moroz and A. Tornquist in \cite{MT} showed that for a large class of non-amenable groups the isomorphism relation for ergodic actions is a complete analytic
  set. This class includes all groups $G$ that have a subgroup isomorphic to the  free group on 2 generators. Extending an early version of their result E. Gardella and M. Lupini
   posted a preprint \cite{GL} in which they show that the same result holds for \textbf{all} non-amenable groups.
   This leaves open the question for amenable groups beyond the integers.

   A group $G$ is called \textbf{indicable} if there is a homomorphism $\pi$ from $G$ to the integers, $\mathbb{Z}$, with $\pi(G) = \mathbb{Z}$. This class contains
  \red{omit: ``of course"}
   $\mathbb{Z}^d$ for all finite $d$ but is a much larger class.
    To see this simply observe that for an arbitrary countable group $H$ the product group $H \times \mathbb{Z}$ is indicable. We will now show the following:

    \begin{theorem} Let $G$ be an indicable group. Then the isomorphism relation for free ergodic actions of $G$ is a complete analytic set.
    \end{theorem}
    \begin{proof}    {Let $\pi:G\to \mathbb{Z}$ witness that $G$ is indicable.} We will construct a Borel reduction of the isomorphism relation for ergodic actions of $\mathbb{Z}$ to those of $G$ and then the proof follows by the main result of
    \cite{FRW}.
     For the construction
   of our reduction we will use a Bernoulli shift action of $G$. Define the space $\Omega= \{0,1\}^G$ and then let $\mu$ be the product measure on this
   space where on each coordinate we have the measure $\{ 1/2 ,1/2 \}$. Finally the shift of action of $G$ is defined by $S_g\omega(k) = \omega(kg)$. It is easy
   to see that this action is free and mixing so that not only is it ergodic but its direct product with any ergodic system is also ergodic.

     Since all standard finite measure spaces are isomophic we can view the space of ergodic actions {of $G$ on the measure space $([0,1],\B, \lambda) \times (\Omega,\B,  \mu)$}.
     We will define a Borel reduction $f$ of the isomorphism relation for $\mathbb{Z}$ to the isomorphism relation for ergodic
     $G$  actions  on the measure space $([0,1],\B, \lambda) \times (\Omega,\B,  \mu)$.
     For each $ T \in MPT$ set $$f(T)_g(x ,\omega) = (T^{\pi(g)}x, \sigma_g(\omega)).$$ By our previous remark if $T$ is ergodic so is $f(T)$. Furthermore
     clearly if $T_1 \cong T_2$ then $f(T_1) \cong f(T_2)$. We now have to show, that conversely, if $f(T_1) \cong f(T_2)$ then $T_1 \cong T_2$.

     Let $H = \textrm{ker}(\pi)$, that is the group of elements $ h \in g$ such that $\pi(h)=0$. By the definition for any $h \in H$ $F(T)_h$ acts as the
     identity in the first coordinate. Thus the $\sigma$-algebra of $H$-invariant subsets for any $f(T)$ is the collection of all sets of the form $B \times \Omega$
     where $ B \in \B$. An isomorphism between the $G$ actions $f(T_1)$ and $f(T_2)$ must preserve this sub-$\sigma$ algebra and therefore defines an isomorphism
     between $T_1$ and $T_2$.  $\Box$
     \end{proof}

\section{Locally compact groups}

  In J.  von Neumann's foundational paper \cite{vN}, where the notion of isomorphism was first defined he was mainly considering not actions of $\mathbb{Z}$ but rather flows, i.e. actions of $\mathbb{R}$.
  In our paper (\cite{FRW} , section 9.3) we considered the isomorphism problem for flows and showed that for ergodic flows the isomorphism relation is a complete analytic set. This was done by
  constructing a Borel reduction from transformations to flows which preserves isomorphisms. Here is a brief description of the reduction.

    To each $T \in MPT$ we define a measure preserving  flow $F(T)_s$ on $[0,1] \times [0,1]$ with the product Lebesgue measure called the time-one suspension of the
    of the transformation $T$. Informally, the flow is defined  by starting from the point $(x,y$) moving  up to $(x, y+s)$ as long as $y+s < 1$. When $y + s =1$ we move to  the point
    $(Tx,0)$ and then continue moving up as before. More formally, we identify
    $(x,1)$ with $(Tx,0)$ for all $x \in [0,1]$ to obtain a  cylinder and then simply flow along the foliation of the cylinder by the ``vertical lines. This function is Borel mesurable, in fact it is contnuous.

      Clearly if $T_1 \cong T_2$ then the flows $F(T_1)$ and $F(T_2)$ are isomorphic. To see the converse notice that even if $T$ is ergodic the transformation $F(T)_1$ is not ergodic,
      and its ergodic components are the horizontal lines $[0,1] \times \{y\}$. On each horizontal line the transformation is the original $T$. Thus since the isomorphism of flows  will preserve this ergodic
      decomposition it follows that isomorphism of the flows implies the isomorphism of the transformations.

     For locally compact groups a reduction similar to what was done for indicable groups can be carried out for groups $G$ that have a continuous surjective homomorphism onto $\mathbb{R}$. This includes such groups
     as $\mathbb{R}^d$ for any integer $d$ and of course any groups of the form $H \times \mathbb{R}$ where $H$ is a locally compact group. We now need some fixed  mixing action
     of the locally compact group $G$. Perhaps the simplest such action is the Poisson point process on $G$ with intensity given by a left invariant Haar measure.
     The group acting on itself by left multiplication will preserve the  probability measure on the sample space on which the Poisson point process is defined. The Borel reduction is
     completely analogous to what we did for discrete groups. 

     Up to now we have been considering probability measure preserving actions. The same question can be raised for ergodic actions preserving an infinite $\sigma$-finite measure.
     This was answered for $\mathbb{Z}$ in \cite{GK} where they show that here too the isomorphism relation is complete analytic. To extend this to indicable groups the same construction can be
     carried out. The only new point to check is that the product action is ergodic. A simple way to see this using only results from the classical theory is the following. Since the homomorphism
     $\pi$ is onto any element $g_0 \notin \textrm{ker}(\pi)$ generates an infinite cyclic group. Now the element $g_0$ acts as a mixing transformations in the Bernoulli shift over $G$. In particular
     it is mildly mixing and therefore its direct product with any infinite ergodic measure preserving transformation is again infinite measure preserving and ergodic by the multiplier property
     of mildly mixing transformations.

 \section{On Halmos's conjugacy problem}

   At the end of his classic introduction, which contains lectures given in 1955,  Paul Halmos \cite{Ha}  lists ten Unsolved Problems. Here is his third problem:

   \begin{quote}
 \textbf{ 3.} The outstanding algebraic problem of ergodic theory is the problem of conjugacy: when are two transformations conjugate? This is vague, of course, but there
   are some interesting and quite specific yes-or-no questions connected with it that should be solved. (The general problem might never be solved; it is not even clear what it means.)
   Here are a few samples: (a) Do there exist measure-theoretically non-isomorphic spaces such that the shifts based on them are conjugate? (b) Do there exist two equivalent but non-conjugate
   transformations with continuous spectrum? (c) Do there exist two conjugate ergodic automorphisms of a compact group that do not belong to the same conjugate class in the group
   of all automorphisms of that group?
   \end{quote}

     His specific sample questions were solved many years ago and it is perhaps worthwhile recalling the answers. The introduction of entropy by A. Kolmogorov just a few years
     later (\cite{K-1}, \cite{K-2}) , gave a  quick answer to question (b), namely the 2-shift and the 3-shift both have countable multiplicity Lebesgue spectrum but are not isomorphic since their entropies
     differ. The first to answer question (a) was  L.D.  Meshalkin \cite{Me} who showed in 1959, inter alia, that the Bernoulli shifts base on $(1/4, 1/4, 1/4, 1/4)$ and $(1/2, 1/8, 1/8, 1/8, 1/8)$ were
     isomorphic. Finally question (c) was answered in 1967 by R.Adler and B. Weiss \cite{AW} who showed that the ergodic automorphisms of the torus were isomorphic if and only if they have the same entropy and thus, for example, the
     automorphisms defined by the matrices:
    $\begin{pmatrix} 5 & 2  \\ 2 & 1
     \end{pmatrix}$
     and
     $\begin{pmatrix} 5 & 4 \\ 1 & 1
     \end{pmatrix}$
      are isomorphic but they are not conjugate elements of $Gl(2, \mathbb{Z})$.

\end{document}